\documentclass[reqno,12pt]{amsart}
\usepackage{amsmath,amsthm,amscd,amssymb,graphicx,enumerate,latexsym}

\usepackage[raiselinks,colorlinks]{hyperref}
\hypersetup{citecolor=blue}

\numberwithin{equation}{section}

\theoremstyle{plain}
\newtheorem{thm}{Theorem}[section]
\newtheorem{lemma}[thm]{Lemma}

\newtheorem{cor}[thm]{Corollary}
\theoremstyle{definition}

\theoremstyle{remark}

\def\Re{\mathop{\rm Re}\nolimits}

\newcommand{\ol}{\overline}

\newcommand{\loc}{\text{\rm{loc}}}

\newcommand{\ac}{\text{\rm{ac}}}

\newcommand{\sing}{\text{\rm{sing}}}

\newcommand{\AC}{\text{\rm{AC}}}

\allowdisplaybreaks

\renewcommand{\MRhref}[2]{\href{http://www.ams.org/mathscinet-getitem?mr=#1}{#2}}
\renewcommand{\MR}[1]{}

\makeatletter
\def\@strippedMR{}
\def\@scanforMR#1#2#3\endscan{%
   \ifx#1M\ifx#2R\def\@strippedMR{#3}%
   \else\def\@strippedMR{#1#2#3}%
   \fi\fi}
\renewcommand\MR[1]{\relax\ifhmode\unskip\spacefactor3000 \space\fi
   \@scanforMR#1\endscan
   \MRhref{\@strippedMR}{MR\@strippedMR}}
\makeatother

\providecommand{\MRhref}[2]{%
  \href{http://www.ams.org/mathscinet-getitem?mr=#1}{#2}
}

\title[Derivatives of $L^p$ eigenfunctions]{Derivatives of $L^p$ eigenfunctions of Schr\"odinger operators}
\author{Milivoje Lukic}
\date\today
\address{6100 Main Street, Rice University, Mathematics MS 136, Houston, TX 77005}
\email{milivoje.lukic@rice.edu}


\begin{document}

\begin{abstract}
Assuming the negative part of the potential is uniformly locally $L^1$, we prove a pointwise $L^p$ estimate on derivatives of eigenfunctions of one-dimensional Schr\"odinger operators. In particular, if an eigenfunction is in $L^p$, then so is its derivative, for $1\le p\le \infty$.
\end{abstract}

\maketitle
\section{Introduction}

In this note we study eigenfunctions $u$ of a one-dimensional Schr\"odinger operator,
\begin{equation}\label{1.1}
- u''(x) + V(x) u(x) = E u(x)
\end{equation}
where $V$ is a real-valued function and $E\in\mathbb{C}$. If $V\in L^1_\loc$, standard existence and uniqueness results for ODEs (see, e.g., Teschl~\cite[Theorem 9.1]{Teschl}) state that \eqref{1.1} has a two-dimensional space of solutions with $u, u' \in \AC_\loc$. Here $\AC_\loc$ stands for the space of functions which are absolutely continuous on  compact intervals.

In spectral theory, $L^p$ properties of solutions of \eqref{1.1} are often of interest; for example, $L^2$ solutions of \eqref{1.1} for $E\in \mathbb{R}$ correspond to eigenvalues of the Schr\"odinger operator $-\frac{d^2}{dx^2} + V(x)$. In many methods, $L^p$ properties of derivatives of solutions are also of importance.

We will prove a pointwise $L^p$ estimate on $u'$, which will provide a proof that $u\in L^p$ implies $u' \in L^p$ under a mild condition on the negative part of $V$. Our estimate will also imply that $u\in L^p$ with $p<\infty$ implies pointwise decay of $u$ and $u'$.

Throughout the paper, the condition on $V$ will be
\begin{equation}\label{1.2}
C_1 = \sup_x \int_x^{x+1} V_-(y) dy < \infty.
\end{equation}
i.e.\ that the negative part of $V$ is uniformly locally $L^1$.

\begin{thm}\label{T1.1}
Let $V\in L^1_\loc$ obey \eqref{1.2}, and let $u(x)$ be a solution of \eqref{1.1} with $E\in \mathbb{C}$.
\begin{enumerate}[(i)]
\item Denoting $C_2 = C_1+\lvert E\rvert$, there exist constants $C = C_2 + 2 \sqrt{C_2}$ and $K = 1 / \sqrt{C_2}$ such that
\begin{equation}\label{1.3}
\lvert u'(x) \rvert \le C \max_{y\in [x-K,x+K]} \lvert u(y) \rvert.
\end{equation}

\item Let $u(x)\neq  0$, $\Re [\ol{u(x)} u'(x)]\ge 0$. Then
\begin{equation}\label{1.4}
\lvert u(y) \rvert >\frac{\lvert u(x)\rvert}2\text{ for }y\in [x,x+\delta),
\end{equation}
where $\delta = - \frac 12 + \sqrt{\frac 14 + \frac 1{2C_2}}$.

\item For $1\le p<\infty$,
\begin{equation}\label{1.5}
\lvert u(x)\rvert^p \le \frac{2^p}\delta \int_{x-\delta}^{x+\delta} \lvert u(y) \rvert^p dy.
\end{equation}

\item For $1\le p<\infty$,
\begin{equation}\label{1.6}
\lvert u'(x) \rvert^p \le  \frac{2^p C^p}\delta \int_{x-K-\delta}^{x+K+\delta} \lvert u(y) \rvert^p dy.
\end{equation}

\item Let $1\le p \le \infty$ and let $w:\mathbb{R}\to (0,\infty)$ obey
\begin{equation}
\sup_{\substack{x, y\in\mathbb{R} \\ \lvert x-y\rvert \le K+\delta}} \frac{w(x)}{w(y)} < \infty.
\end{equation}
Then $u \in L^p(w(x)dx)$ implies $u'\in L^p(w(x)dx)$.

\item If $u\in L^p(dx)$ with $p<\infty$, then
\begin{equation}
\lim_{x\to \pm\infty} u(x) = \lim_{x\to\pm\infty} u'(x) = 0.
\end{equation}
\end{enumerate}
\end{thm}

Results of this type have appeared in the literature as technical lemmas; Stolz proved Theorem~\ref{T1.1}(v) for some weighted $L^2$ spaces in \cite[Proposition 8]{Stolz95}, and for $L^\infty$ with $E\in \mathbb{R}$ in \cite[Lemma 4]{Stolz92}.

Pointwise estimates of the type \eqref{1.6} have appeared before in the literature; Simon~\cite[Lemma 3.1]{Simon} proves such a bound for $p=2$, under the stronger condition that $V$ be uniformly locally $L^2$. 

In Section~\ref{S2}, we discuss some applications of Theorem~\ref{T1.1} to the spectral theory of Schr\"odinger operators.
In Section~\ref{S3}, we present the proof of Theorem~\ref{T1.1}.

It is my pleasure to thank David Damanik, Fritz Gesztesy, Barry Simon and G\"unter Stolz for useful discussions.

\section{Applications to spectral theory}\label{S2}

We present some applications of these estimates to spectral theory. These are not new results, but estimates of Theorem~\ref{T1.1} are relevant to their proofs. These are half-line results, so in this section, $H=-\frac{d^2}{dx^2} + V$ will be the Schr\"odinger operator on $(0,+\infty)$. We assume $0$ is a regular point, i.e. $V\in L^1(0,1)$, so $u(x)$ and $u'(x)$ have finite limits as $x\to 0$.

Our first application is to an alternative proof that bounded eigenfunctions imply absolutely continuous spectrum. We are referring to the following theorem.

\begin{thm}\label{T2.1}
Let $V\in L^1_\loc$ be a half-line potential with a regular point at $0$ which obeys \eqref{1.2} and let
\[
S = \{E \in \mathbb{R} \vert \text{solutions of \eqref{1.1} are bounded on }[0,\infty)\}.
\]
Then the spectral measure $\mu$ of $H=-\frac{d^2}{dx^2}+V(x)$ obeys
\begin{enumerate}[(i)]
\item $\mu_\sing(S) = 0$;
\item $\mu_\ac(T)>0$ for any $T\subset S$ with $\lvert T\rvert>0$ (where $\lvert \cdot \rvert$ is the Lebesgue measure).
\end{enumerate}
\end{thm}

This theorem was first proved by Behncke~\cite{Behncke} and Stolz~\cite{Stolz92}, who proved that \eqref{1.2} and boundedness of eigenfunctions for $E\in S$ allows one to use the subordinacy theory of Gilbert--Pearson~\cite{GilbertPearson} to imply the conclusions of the above theorem.

A more direct proof was found by Simon~\cite{Simon}. However, the proof in \cite{Simon} assumes that $V$ is uniformly locally $L^2$ in order to bound $u'$ locally in terms of $u$. Replacing that part of the argument by  \eqref{1.6}, the proof in \cite{Simon} generalizes to all potentials $V$ included by Theorem~\ref{T2.1}. It should be noted that this method needs the estimate \eqref{1.6} for non-real energies $E$, which Theorem~\ref{T1.1} provides.

In the remainder of this section, we point out some simple criteria for point spectrum. These criteria use the implication
\begin{equation}\label{2.1}
u \in L^2 \implies u' \in L^2.
\end{equation}
This is a special case of Theorem~\ref{T1.1}(v), but we remind the reader that it was previously proved by Stolz~\cite[Proposition 8]{Stolz95}.

Simon--Stolz~\cite{SimonStolz} provide a criterion for absence of eigenvalues in terms of transfer matrices. The transfer matrix $T(E,x,y)$ is defined by
\[
T(E,x,y) \begin{pmatrix} u(y) \\ u'(y) \end{pmatrix} =  \begin{pmatrix} u(x) \\ u'(x) \end{pmatrix}.
\]
for solutions $u$ of \eqref{1.1}. The Simon--Stolz criterion uses the condition
\begin{equation}\label{2.2}
\int_0^\infty \frac{dx}{\lVert  T(E,x,0) \rVert^2} = \infty
\end{equation}
to prove that \eqref{1.1} has no $L^2$ solution.  Their theorem also assumes $V$ is bounded from below, but their proof, combined with the implication \eqref{2.1}, gives

\begin{cor}
Let $V\in L^1_\loc$ be a half-line potential with a regular point at $0$ which obeys \eqref{1.2} and let $E\in \mathbb{R}$ be such that \eqref{2.2} holds. Then $-\Delta+V$, as a Schr\"odinger operator on $L^2(\mathbb{R}^+)$, doesn't have an eigenvalue at $E$.
\end{cor}

\begin{proof}
The argument of Simon--Stolz~\cite[Theorem 2.1]{SimonStolz} goes unchanged to prove $\left \lVert \begin{pmatrix} u(x) \\ u'(x) \end{pmatrix}\right \rVert \notin L^2$ for any solution of \eqref{1.1}. 
\eqref{2.1} then implies $u\notin L^2$, so $E$ is not an eigenvalue of $-\Delta+V$.
\end{proof}

For a real-valued non-zero solution of \eqref{1.1} and $E=k^2>0$, Pr\"ufer variables are defined by
\begin{align*}
u'(x) & = k R_k(x) \cos \theta_k(x) \\
u(x) & = R_k(x) \sin \theta_k(x)
\end{align*}
with $R_k(x)>0$, $\theta_k(x) \in \mathbb{R}$. They were first introduced by Pr\"ufer \cite{Prufer} and have found extensive use in spectral theory, see e.g.\ Kiselev--Last--Simon~\cite{KiselevLastSimon}.
Note that
\begin{equation}\label{2.3}
k^2 R_k(x)^2 = u'(x)^2 + k^2 u(x)^2.
\end{equation}
The following corollary is immediate from \eqref{2.3} and \eqref{2.1}.

\begin{cor}
Let $V\in L^1_\loc$ be a half-line potential with a regular point at $0$ which obeys \eqref{1.2} and let $E=k^2>0$. Then $u\in L^2$ if and only if $R_k \in L^2$.
\end{cor}

\section{Proof  of Theorem~\ref{T1.1}}\label{S3}

The basis of all the estimates will be the following inequality:

\begin{lemma}\label{L3.1}
Let $x<y$ and assume $\omega \in \mathbb{C}$, $u(x)\neq 0$, and $\Re[ \bar \omega u(t)] \ge 0$ for $t \in [x,y]$. Then
\begin{equation}\label{3.1}
\Re[\bar \omega u(y)] \ge \Re[\bar\omega u(x)] + (y-x) \Re[ \bar\omega u'(x)] - C_2  (y-x) (y-x+1)\lvert \omega \rvert \max_{x\le t \le y} \lvert u(t) \rvert
\end{equation}
\end{lemma}

\begin{proof}
Using absolute continuity of $u$ and $u'$,
\begin{align}
u(y) & = u(x) + \int_x^y \left[u'(x) + \int_x^t u''(s) ds \right] dt \nonumber \\
& = u(x) + (y-x) u'(x) +  \int_x^y (y-s) u''(s) ds \label{3.2}
\end{align}
Denoting $M =  \max_{x\le t \le y} \lvert u(t)\rvert$, we have $0 \le \Re[\bar \omega u(s)] \le \lvert \bar \omega u(s) \rvert  \le \lvert \omega \rvert M$ for $s \in [x,y]$, so by $u'' = V u - Eu$,
\begin{align*}
\Re \left[ \bar \omega \int_x^y (y-s) u''(s) ds \right]  & = \int_x^y (y-s)  V(s) \Re \left[ \bar \omega  u(s)  \right] ds - \int_x^y (y-s) \Re \left[ \bar \omega  E u(s) \right] ds   \\
& \ge - \lvert \omega \rvert M (y-x) \int_x^y V_-(s) ds - \lvert \omega E \rvert M (y-x)^2 \nonumber \\
& \ge - \lvert \omega \rvert M (y-x) (y-x+1) (C_1 + \lvert E\rvert) 
\end{align*}
which together with \eqref{3.2} proves \eqref{3.1}.
\end{proof}

\begin{proof}[Proof of Theorem~\ref{T1.1}]
(i)  Without loss of generality, assume $\Re[\ol{u(x)} u'(x)]\ge 0$ (the other case follows by considering $u(-x)$).

Let $M = \max_{x-K \le y \le x+K} \lvert u(y) \rvert$. Assume that, contrary to \eqref{1.3}, we have \begin{equation}\label{3.3}
\lvert u'(x) \rvert > C_2 (1+2 K) M.
\end{equation}
Denote $f(y)=\Re[\ol{u'(x)} u(y)]$. Applying Lemma~\ref{L3.1} with $\omega = u'(x)$, we have
\begin{align}
f(y) & \ge f(x) + (y-x) \lvert u'(x)\rvert^2 - C_2 (y-x) (y-x+1) \lvert u'(x) \rvert M  \nonumber \\
& \ge M (y-x) \lvert u'(x)\rvert \bigl(\lvert u'(x)\rvert  - C_2 (y-x+1) \bigr) \label{3.4}
\end{align}
for $y\in [x,x+K]$ such that $f(t) \ge 0$ for all $t\in [x,y]$. 

Note that $f$ is continuous, $f(x)\ge 0$ and $f'(x) = \Re[\ol{u'(x)} u'(x)]>0$, so $f>0$ in some interval $(x,x+\epsilon)$.  We claim that $f>0$ in $(x,x+K]$; assume to the contrary, that there exists $y\in (x,x+K]$ such that $f(y)=0$, and pick the smallest such $y$. Then $f\ge 0$ on $[x,y]$, so by \eqref{3.4} and \eqref{3.3},
\begin{equation}\label{3.5}
f(y) > M (y-x) \lvert u'(x) \rvert C_2 (2K - (y-x)) > 0
\end{equation}
contradicting our assumption and proving $f>0$ on $(x,x+K]$. Taking $y= x+ K$ in \eqref{3.5}, we have
\[
\Re[\ol{u'(x)} u(x+K)]  >  C_2 M K^2 \lvert u'(x) \rvert = M \lvert u'(x) \rvert \ge \lvert \ol{u'(x)} u(x+K) \rvert
\]
which is a contradiction. Thus, the initial assumption \eqref{3.3} is wrong.

(ii) Assume the contrary; then there exists $y \in (x,x+\delta)$ such that $\lvert  u(y)\rvert =\frac{\lvert u(x)\rvert}2$. Let $z\in [x,y)$ be such that
\[
\lvert u(z)\rvert = \max_{t \in [x,y]} \lvert u(t)\rvert.
\]
Since $\Re[\ol{u(t)} u'(t)]=\frac 12 \frac d{dt} \lvert u(t) \rvert^2$, we have $\Re[\ol{u(z)} u'(z)] = 0$ (this is true even if $z=x$ since we know a priori that $\Re[\ol{u(x)} u'(x)] \ge 0$). Note also
\begin{equation}\label{3.6}
\Re[ \ol{u(z)} u(y) ] \le \lvert \ol{u(z)}  u(y)\rvert \le \frac{\lvert u(z)\rvert^2}2,
\end{equation}
so we may pick $t \in (z,y]$ as the smallest $t>z$ with $\Re[ \ol{u(z)} u(t) ] = \frac{\lvert u(z) \rvert^2}2$. 

Using (i) with $x$ replaced by $z$ and $y$ replaced by $t$, and with $\omega = u(z)$ gives
\begin{align*}
\Re [ \ol{u(z)} u(t) ] & \ge \lvert u(z)\rvert^2 [ 1 - C_2 (t-z)(t-z+1)]  \\
& > \lvert u(z)\rvert^2 [1 - C_2 \delta (\delta+1)] \\
& = \frac{\lvert u(z)\rvert^2}2
\end{align*}
where we used $t-z\le y-x < \delta$. This is a contradiction with \eqref{3.6}, which completes the proof.

(iii) For $\Re[ \ol{u(x)} u'(x)] \ge 0$, the claim follows directly from (ii) by 
taking the $p$-th power of \eqref{1.4} and integrating from $x$ to $x+\delta$.
The case $\Re[ \ol{u(x)} u'(x)] < 0$ follows by symmetry, by considering $u(- x)$.

(iv) This follows directly from (i) and (iii).

(v) We start with \eqref{1.3} for $p=\infty$ or \eqref{1.6} for $p<\infty$, and multiply by $w(x) \le C_3 w(y)$. For $p<\infty$, integrating in $x$ and using Tonelli's theorem completes the proof.

(vi) If $u\in L^p$ with $p<\infty$, then the right hand sides of \eqref{1.5}, \eqref{1.6} converge to $0$ as $x\to \pm \infty$, so the left hand sides also converge to $0$.
\end{proof}

\bibliographystyle{amsplain}

\end{document}